\documentclass{amsart}
\usepackage{amssymb}
\usepackage{stmaryrd}
\usepackage{url}
\usepackage{enumitem}

\newcommand{\ZZ}{\mathbf{Z}}
\newcommand{\QQ}{\mathbf{Q}}
\newcommand{\FF}{\mathbf{F}}
\newcommand{\RR}{\mathbf{R}}
\newcommand{\norm}[1]{\Vert #1 \Vert}
\newcommand{\eps}{\varepsilon}

\newtheorem{thm}{Theorem}
\newtheorem{prop}[thm]{Proposition}
\newtheorem{lem}[thm]{Lemma}
\theoremstyle{remark}

\begin{document}

\title{Counting points on hyperelliptic curves in average polynomial time}
\author{David Harvey}
\address{School of Mathematics and Statistics, University of New South Wales, Sydney NSW 2052, Australia}
\email{d.harvey@unsw.edu.au}
\urladdr{http://web.maths.unsw.edu.au/~davidharvey/}

\begin{abstract}
Let $g \geq 1$ and let $Q \in \ZZ[x]$ be a monic, squarefree polynomial of degree $2g + 1$. For an odd prime $p$ not dividing the discriminant of $Q$, let $Z_p(T)$ denote the zeta function of the hyperelliptic curve of genus $g$ over the finite field $\FF_p$ obtained by reducing the coefficients of the equation $y^2 = Q(x)$ modulo $p$. We present an explicit deterministic algorithm that given as input $Q$ and a positive integer $N$, computes $Z_p(T)$ simultaneously for all such primes $p < N$, whose average complexity per prime is polynomial in $g$, $\log N$, and the number of bits required to represent $Q$.
\end{abstract}

\maketitle

\dedicatory{\it For my wife, Lara}

\section{Introduction}
\label{sec:introduction}

A central problem in computational arithmetic geometry is to give efficient algorithms for the calculation of the zeta function of a variety $X$ over a finite field $\FF_q$, where $q = p^a$. The zeta function of $X$ is the generating function
 \[ Z_X(T) = \exp \left(\sum_{n \geq 1} \frac{\# X(\FF_{q^n})}n T^n\right) \in \ZZ\llbracket T\rrbracket. \]
Dwork proved that $Z_X(T)$ is a rational function, so to compute it means to explicitly find its numerator and denominator as polynomials. More background on the algorithmic theory of zeta functions may be found in the survey article \cite{Wan-zeta}.

In this paper we focus on the specific case of a hyperelliptic curve $X$ of genus $g \geq 1$, with a rational Weierstrass point. Assuming $p \neq 2$, such a curve is given by an equation $y^2 = Q(x)$ where $Q \in \FF_q[x]$ is monic and squarefree, of degree $2g+1$. The zeta function has the form
 \[ Z_X(T) = \frac{P(T)}{(1-T)(1-qT)}, \]
where $P \in \ZZ[T]$ has degree $2g$.

In this situation, there are many algorithms known for computing $Z_X(T)$. One family derives from Schoof's algorithm for elliptic curves \cite{Sch-elliptic, Pil-abelian, AH-counting}. These $\ell$-adic algorithms achieve time complexity $(\log q)^{g^{O(1)}}$, which for fixed genus is polynomial in $\log p$ and $a$, but in general is exponential in $g$. (In this paper, time complexity always means bit complexity in the sense of the multitape Turing model \cite{Pap-complexity}.) These algorithms have been successfully deployed in genus one and two --- see \cite{Sut-modular} and \cite{GS-genus2} for recent record computations --- but the author is aware of no attempts for $g \geq 3$.

The $p$-adic algorithms form a much more diverse family. These all have the drawback that the complexity is exponential in $\log p$. One example, highly relevant to the present work, is Kedlaya's algorithm \cite{Ked-hyperelliptic}, which has complexity $p^{1+\eps} a^{3+\eps} g^{4+\eps}$. Here and below, $Y^\eps$ means $Y^{o(1)}$, where $o(1)$ is a quantity approaching zero as $Y \to \infty$. The exponent of $p$ can be improved to $p^{1/2 + \eps}$ at the expense of increasing the exponents of $a$ and $g$ \cite{Har-kedlaya}, but this is still exponential in $\log p$.

The main open problem in this area is whether there exists an algorithm whose complexity is simultaneously polynomial in $g$ and $\log q$. In other words, we ask for an algorithm whose complexity is polynomial in the size of the input. The latter is $\Theta(g \log q)$, the number of bits required to represent $Q(x)$.

In this paper we prove a weaker result in this direction, namely that it is possible to achieve polynomial time complexity \emph{on average over $p$}. We consider the following situation. Let $Q \in \ZZ[x]$ be a monic, squarefree polynomial of degree $2g + 1 \geq 3$. Let $X$ be the hyperelliptic curve of genus $g$ over $\QQ$ defined by $y^2 = Q(x)$, i.e.~the normalisation of the projective closure of the affine curve. For any odd prime $p$ not dividing the discriminant of $Q(x)$, let $\overline X_p$ be the hyperelliptic curve of genus $g$ over $\FF_p$ defined by the same equation $y^2 = Q(x)$, but with coefficients reduced modulo $p$. Let $\norm{Q}$ denote the maximum of the absolute values of the coefficients of $Q$.
\begin{thm}
\label{thm:main}
There exists an explicit deterministic algorithm with the following properties. The input consists of integers $N \geq 3$, $g \geq 1$, and a polynomial $Q \in \ZZ[x]$ defining a hyperelliptic curve $X$ of genus $g$ as above. The output is the sequence of zeta functions of $\overline X_p$, for all odd primes $p < N$, with $p$ not dividing the discriminant of $Q$. The algorithm runs in
 \[ g^{8+\varepsilon} N \log^2 N \log^{1+\eps} (\norm Q N) \]
bit operations.
\end{thm}
Since the number of primes $p < N$ is asymptotically $N/\log N$, the average time per prime is
 \[ g^{8+\eps} \log^3 N \log^{1+\eps} (\norm Q N), \]
which is polynomial in the size of the input.

One obvious application of this result is to the computation of $L$-series of hyperelliptic curves over $\QQ$, with a view towards collecting numerical data on questions such as the Birch--Swinnerton-Dyer conjecture and the Sato--Tate conjecture for these curves. Such investigations have recently been carried out by Fit\'e, Kedlaya, Rotger and Sutherland for curves of genus up to three \cite{KS-L-series, KS-hyperelliptic, FKRS-genus2}, with particularly detailed information being obtained for genus two curves. In this context it is reasonable to assume that the coefficients are small relative to $N$, say $\log \norm Q = O(\log N)$, so that the average time per prime is simply  $g^{8+\eps} \log^{4+\eps} N$. The new algorithm may make it possible to dramatically extend the range of their numerical results.

In fact, even in the case of elliptic curves, Theorem \ref{thm:main} already yields the best known unconditional complexity bound for computing the trace of Frobenius for all $p < N$ simultaneously. Previously, the best known unconditional deterministic bound was $\log^{5+\eps} p$ per prime, achieved by Schoof's original algorithm (see \cite[p.~111]{BSS-ellcrypt}). The Schoof--Elkies--Atkin (SEA) algorithm is conjectured to improve this (probabilistically) to $\log^{4+\eps} p$. For more information about the heuristics involved in the latter estimate, see the discussion preceding Theorem 13 of \cite{Sut-modular}.

It is likely that this theorem can be extended in several ways. First, the restriction to curves with a rational Weierstrass point is inherited from \cite{Ked-hyperelliptic} and \cite{Har-kedlaya}; it surely can be lifted, along the lines of \cite{Har-extension}. Second, the method should extend to superelliptic curves, following \cite{GG-superelliptic, Min-superelliptic}. Third, it should be possible to apply the same method to a hyperelliptic curve defined over a number field $K$. The resulting complexity bound should depend polynomially on $a = [K:\QQ]$, and also on the size of the coefficients of a defining polynomial for $K/\QQ$.

Our starting point for the new algorithm is the author's modification of Kedlaya's algorithm \cite{Har-kedlaya}. The portion of this algorithm whose complexity is exponential in $\log p$ involves computing various `reduction matrices'. These are products of the form $M_p(1) M_p(2) \cdots M_p(p)$, where $M_p(x)$ is a matrix of size $O(g)$ whose entries are linear polynomials in $x$ over $\ZZ_p$. In that paper we suggested using the method of \cite{BGS-recurrences} to evaluate this product using $g^3 p^{1/2+\eps}$ ring operations in $\ZZ_p$.

A key observation is that such products may enjoy a certain redundancy: for $p_1 < p_2$, the product $M_{p_1}(1) \cdots M_{p_1}(p_1)$ may be a subproduct of $M_{p_2}(1) \cdots M_{p_2}(p_2)$. To realise any advantage from this, we must overcome two obvious obstructions.

The first is that the values lie in different rings; there is no relation between $\QQ_{p_1}$ and $\QQ_{p_2}$ for $p_1 \neq p_2$. We will deal with this by evaluating the products over $\QQ$ rather than $\QQ_p$. It would appear that coefficient explosion renders this approach woefully inefficient. Coefficient growth does indeed occur, and one of our key tasks is to bound it.

The second, more fundamental obstruction, is that the entries of $M_p(x)$ might depend on $p$, as suggested by the notation. This does in fact occur in the `horizontal reductions' of \cite{Har-kedlaya}, via the dependence on $t$ in \cite[\S7.2]{Har-kedlaya}. The first clue towards removing this dependence is the observation that the `vertical reduction' matrices of \cite{Har-kedlaya} do \emph{not} depend on $p$. The difference is that these matrices `reduce towards zero', in a sense that will be made clear in Section \ref{sec:reductions}. Therefore our solution is to revisit the definition of the relevant cohomology spaces, and design a reduction strategy that `reduces towards zero' in all cases. This leads to reduction matrices $M(x)$ whose entries depend only on the coefficients of $Q(x)$, and not on $p$. The problem of simultaneous zeta function computation is thus transformed into the problem of computing products of the form $M(1) M(2) \cdots M(p)$, modulo a suitable power of $p$, simultaneously for all $p < N$.

For this, we leverage recent work on the computation of Wilson quotients, or equivalently the residues $u_p = (p-1)! \pmod{p^2}$. The best known algorithm for computing a single $u_p$ has complexity $p^{1/2+\eps}$. For computing the $u_p$ in bulk, the paper \cite{CGH-wilson} introduced an ``accumulating remainder tree'' technique that computes $u_p$ for all $p < N$ simultaneously in $N \log^{3+\eps} N$ bit operations; that is, in average polynomial time per prime. The accumulating remainder tree succeeds in reconciling two conflicting algorithm design strategies: on one hand, we wish to work modulo $p^2$ to avoid the growth of the factorials; on the other hand, we want to exploit redundancies in the products $(p-1)!$ for varying $p$. This conflict is exactly what we face for the matrix $M(x)$ discussed above. In this paper we adapt the accumulating remainder tree to the matrix case, replacing the linear polynomial $x$ by $M(x)$, to compute the products $M(1) \cdots M(p)$, modulo an appropriate power of $p$, in average polynomial time per prime.

\section{Preliminaries}
\label{sec:preliminaries}

For the rest of the paper we fix the following notation. We try to follow the notation of \cite{Ked-hyperelliptic} and \cite{Har-kedlaya} as closely as possible, with additional decoration to keep track of the dependence on $p$.

As in Theorem \ref{thm:main}, we take a hyperelliptic curve $X$ given by the equation $y^2 = Q(x)$ where $Q \in \ZZ[x]$ is monic and squarefree, and $\deg Q = 2g + 1$. We denote by $X'$ the curve obtained from $X$ by removing the point at infinity and the Weierstrass points. It is affine, with coordinate ring
 \[ A = \QQ[x, y, y^{-1}]/(y^2 - Q(x)). \]
Elements of $A$ may be represented as finite sums
 \[ f = \sum_{i \geq 0, \, j \in \ZZ} a_{i,j} x^i y^{-j}, \qquad a_{i,j} \in \QQ. \]

Let $\Omega$ be the $A$-module of differential forms on $X'$. This is the module generated by symbols $du$ for $u \in A$, subject to the relations $d(uv) = u \, dv + v \, du$ for $u, v \in A$, and $du = 0$ for $u \in \QQ$. Since $dy = \frac12 Q'(x) dx/y$, elements of $\Omega$ may be represented as finite sums
 \[ \omega = \sum_{i \geq 0, \, j \in \ZZ} a_{i,j} x^i y^{-j} dx/y, \qquad a_{i,j} \in \QQ. \]
Let $\Omega^-$ be the $(-1)$-eigenspace for the hyperelliptic involution $(x, y) \mapsto (x, -y)$. Its elements are finite sums as above, with $a_{i,j} \neq 0$ only for even $j$.

Two forms $\omega_1, \omega_2 \in \Omega$ are cohomologous if $\omega_1 - \omega_2 = df$ for some $f \in A$, and in this case we write $\omega_1 \sim \omega_2$. Using the same method as in \cite{Ked-hyperelliptic}, it can be shown that every $\omega \in \Omega^-$ is cohomologous to a unique $\omega' = \sum_{i=0}^{2g-1} \lambda_i x^i dx/y$ with $\lambda_i \in \QQ$, called the \emph{reduction} of $\omega$.

Now let $p$ be an odd prime of good reduction for $X$, i.e.~such that $p$ does not divide the discriminant of $Q$. We denote by $\overline{X'_p}$ the affine curve over $\FF_p$ with coordinate ring
 \[ \overline A_p = \FF_p[x, y, y^{-1}]/(y^2 - \overline Q_p(x)), \]
where $\overline Q_p \in \FF_p[x]$ is the reduction of $Q$ modulo $p$. Let
 \[ A_p = \ZZ_p[x, y, y^{-1}]/(y^2 - Q_p(x)), \]
where $Q_p \in \ZZ_p[x]$ is the image of $Q$, and let $A_p^\dagger$ be the weak completion of $A_p$, in the sense of Monsky--Washnitzer \cite{MW-formal-I}. Define $\Omega_p$ to be the $A_p^\dagger$-module of differential forms over $\QQ_p$ (i.e.~generated by $du$ for $u \in A_p^\dagger \otimes_{\ZZ_p} \QQ_p$, with the same relations as before), and let $\Omega_p^-$ be its $(-1)$-eigenspace. Two forms $\omega_1, \omega_2 \in \Omega_p$ are cohomologous if $\omega_1 - \omega_2 = df$ for some $f \in A_p^\dagger \otimes_{\ZZ_p} \QQ_p$. The quotient of $\Omega_p$ by this relation is by definition the first Monsky--Washnitzer cohomology group $H^1(\overline {X'_p}; \QQ_p)$, a vector space over $\QQ_p$. We are mainly interested in $V_p = H^1(\overline {X'_p}; \QQ_p)^-$, the subspace corresponding to $\Omega_p^-$. It has dimension $2g$, with basis $\{x^i dx/y\}_{i=0}^{2g-1}$. In other words, every $\omega \in \Omega^-_p$ is cohomologous to a unique $\omega' = \sum_{i=0}^{2g-1} \lambda_i x^i dx/y$ with $\lambda_i \in \QQ_p$, again called the reduction of $\omega$. The two notions of reduction are compatible with the obvious natural map $\Omega^- \to \Omega^-_p$.

Let $\sigma_p : \overline A_p \to \overline A_p$ be the Frobenius map $u \mapsto u^p$. The essence of Kedlaya's method is to give an explicit expression for a lift $\sigma_p : A_p^\dagger \to A_p^\dagger$, and then to calculate the matrix of its action on $V_p$ with respect to the basis given above. The numerator $P(T)$ of the zeta function of $\overline X_p$ is then simply the characteristic polynomial of this matrix. The Weil conjectures provide bounds on the coefficients of this polynomial, so it can be recovered exactly, provided we compute the matrix to sufficiently high $p$-adic precision.

Already here there is a subtle difference with \cite{Ked-hyperelliptic}. In Kedlaya's situation, the input is a curve over $\FF_p$, and he lifts it arbitrarily to $\ZZ_p$. In our case, we begin with a curve over $\QQ$, and we are considering the reductions modulo $p$ for all $p$ simultaneously. It is crucial for our method that we use the `same lift' for all $p$.

The precise definition of $\sigma_p$ is not so important for us (see \cite{Ked-hyperelliptic} for details). The only information we need is the following description of the action of $\sigma_p$ on the basis elements $x^i dx/y$:
\begin{prop}
\label{prop:frob}
Let $\mu \geq 1$, and assume that $p > (2\mu - 1)(2g+1)$. Let $C_{j,r} \in \ZZ$ denote the coefficient of $x^r$ in $Q(x)^j$. For $0 \leq j < \mu$, let
 \[ \alpha_j = \sum_{k=j}^{\mu - 1} (-1)^{j+k} \binom{-1/2}{k} \binom{k}{j} \in \ZZ[\textstyle \frac12]. \]
For $a, b \geq 1$, with $b$ odd, let $U^{a,b}_p$ denote the reduction of $x^{pa - 1} y^{-pb+1} dx/y  \in \Omega^-$.

Then for $0 \leq i < 2g$, the reduction of $\sigma_p(x^i dx/y)$ agrees modulo $p^\mu$ with the image in $\Omega_p^-$ of
 \[ \sum_{j=0}^{\mu - 1} \sum_{r=0}^{(2g+1)j} p \alpha_j C_{j,r} U^{i+r+1,2j+1}_p \]
(i.e.~the coefficients with respect to the basis $\{x^i dx/y\}_{i=0}^{2g-1}$ agree modulo $p^\mu$).
\end{prop}
\begin{proof}
This is just a restatement of \cite[Prop.~4.1]{Har-kedlaya}, taking into account that reduction respects the map $\Omega^- \to \Omega^-_p$.
\end{proof}

The point of this result is that to compute the zeta functions of $\overline X_p$ for many $p$ simultaneously, it will suffice to compute, for finitely many pairs $(a, b)$, the reductions of $x^{pa - 1} y^{-pb+1} dx/y$, modulo a suitable power of $p$, for many $p$ simultaneously. We will return to this in Section \ref{sec:main}.

Note that the hypothesis $p > (2\mu - 1)(2g+1)$ is not stated explicitly in \cite[Prop.~4.1]{Har-kedlaya}, but is a standing assumption for that whole paper; see \cite[Thm.~1.1]{Har-kedlaya}. The original purpose of this assumption was to simplify analysis of denominators. Indeed, the algorithm of \cite{Har-kedlaya}, and the statement of Proposition \ref{prop:frob} above, can be modified to work for smaller primes, but this requires increasing the number of terms in the sum, and carrying more working $p$-adic digits in the algorithm. On the other hand, in the present paper, we are in effect forced to use the same $p$-adic precision for all primes. Therefore this hypothesis now acquires an efficiency implication: to get away with the minimum possible working precision, we must restrict to those primes $p > (2\mu - 1)(2g+1)$.

It will be important to keep track of the size of various objects in our discussion. For a polynomial $f$ with integer coefficients, define $\norm{f}$ to be the maximum of the absolute values of its coefficients. If $M$ is a matrix with integer entries, define $\norm{M} = \max_j \sum_i |M_{ij}|$, i.e.~the maximum of the $L^1$ norms of the columns of $M$. This norm is submultiplicative with respect to matrix multiplication, because
\[ \norm{MN}
   \leq \max_j \sum_i \sum_k |M_{ik}| |N_{kj}|
   \leq \max_j \sum_k |N_{kj}| \max_\ell \sum_i |M_{i\ell}|
   = \norm{N} \norm{M}. \]

We will freely use the following well-known complexity results. Integers with at most $n$ bits may be multiplied in $n \log^{1+\eps} n$ bit operations via fast Fourier transform methods, and division with remainder of integers with at most $n$ bits has the same asymptotic cost \cite[Ch.~8--9]{vzGG-compalg}. Matrices of size $n$ over a ring $R$ may be multiplied using $O(n^3)$ ring operations (but see the comments following the proof of Proposition \ref{prop:tree}). We denote the set of such matrices by $M_n(R)$. The primes less than $N$ may be enumerated in $N \log^{2+\eps} N$ bit operations. Note that the usual complexity bound for the sieve of Eratosthenes is not valid in the Turing model; see \cite[Prop.~4]{CGH-wilson} for a discussion and a proof of the bound given.

We also require a deterministic algorithm for solving certain Bezout equations over $\ZZ[x]$. The literature on this problem focuses on probabilistic algorithms. For lack of a suitable reference, we provide the following result. Our method is quite standard; see for example \cite{vzGG-compalg}.
\begin{lem}
\label{lem:bezout}
Let $F, G \in \ZZ[x]$ be nonzero and relatively prime. Let $m = \deg F$, $n = \deg G$. Let $\delta \in \ZZ$ be the resultant of $F$ and $G$, so $\delta \neq 0$. Then there exist polynomials $R_i, S_i \in \ZZ[x]$, for $0 \leq i < m + n$, with the following properties.
\begin{enumerate}[label=\textup{(\alph*)}]
\item $F R_i + G S_i = \delta x^i$.
\item $\deg R_i < n$ and $\deg S_i < m$.
\item $\log|\delta|$, $\log \norm{R_i}$ and $\log \norm{S_i}$ are all in $O((m + n)\log((m + n)\norm{F} \norm{G}))$.
\item We may compute $\delta$, and all $R_i$ and $S_i$, in
 \[ (m + n)^{3+\eps} \log^{1+\eps} (\norm{F} \norm{G}) \]
bit operations.
\end{enumerate}
\end{lem}
\begin{proof}
Let $P_k$ denote the space of polynomials in $\ZZ[x]$ of degree less than $k$. Let $T$ be the matrix of the map $P_n \times P_m \to P_{m+n}$ given by $(R, S) \mapsto FR + GS$, i.e.~the $(m+n)\times(m+n)$ Sylvester matrix
 \[ T = \begin{pmatrix}
        F_0    &        &        & G_0    &        &        \\
        F_1    &        &        & G_1    &        &        \\
        \vdots & \ddots & F_0    & \vdots & \ddots & G_0    \\
        F_m    &        & F_1    & G_n    &        & G_1    \\
               &        & \vdots &        &        & \vdots \\
               &        & F_m    &        &        & G_n    \\
\end{pmatrix}, \]
where $F_j$ and $G_j$ denote the coefficients of $F$ and $G$. By definition $\delta = \det T$, and by Cramer's rule the coefficients of $R_i$ and $S_i$ are given by certain principal minors of $T$. This proves (a) and (b), and (c) follows by applying the Hadamard bound to each determinant.

We now sketch an algorithm that proves (d). We say that a prime $p$ is `bad' if it divides $\delta$ or the leading coefficients of $F$ or $G$; otherwise it is `good'. The product of the bad primes is certainly at most $|\delta| \norm{F} \norm{G}$. By (c) we may choose $\beta$ with $\beta = O((m + n)\log((m + n)\norm F \norm G))$ so that we are guaranteed $\log \max(|\delta|, \norm{R_i}, \norm{S_i}) \leq \beta$. Increasing $\beta$ by $\log(|\delta| \norm{F} \norm{G}) + O(1) = O((m + n)\log((m + n)\norm F \norm G))$, and using the estimate $\sum_{p < \beta} \log p \sim \beta$, we may ensure that the product $J$ of the good primes less than $\beta$ is large enough so that knowledge of $\delta, R_i, S_i$ modulo $J$ determines $\delta, R_i, S_i$ precisely over $\ZZ$.

Now perform the following steps. Compute the images of $F$ and $G$ in $\FF_p[x]$ for all $p < \beta$. This costs $(m+n) \beta^{1+\eps}$ bit operations using a remainder tree \cite{Ber-fastmult}. For each $p < \beta$, we may determine if $p$ is good, and if so, find polynomials $\overline{R_0}, \overline{S_0} \in \FF_p[x]$ such that $F \overline{R_0} + G \overline{S_0} = \delta \pmod p$, $\deg \overline{R_0} < n$, $\deg \overline{S_0} < m$, in $(m + n)^{1+\eps} \log^{1+\eps} p$ bit operations \cite[Thm.~11.7, Cor.~11.16]{vzGG-compalg}. For $i = 1, \ldots, m + n - 1$, compute $\overline{R_i} = x\overline{R_{i-1}} \bmod G$ and $\overline{S_i} = x\overline{S_{i-1}} \bmod F$, in $(m + n)\log^{1+\eps} p$ bit operations. Then $F \overline{R_i} + G \overline{S_i} = \delta x^i \pmod p$ and $\deg \overline{R_i} < n$, $\deg \overline{S_i} < m$. The cost over all $i$ is $(m+n)^2 \log^{1+\eps} p$, so over all $p < \beta$ is $(m + n)^2 \beta^{1+\eps}$ bit operations. Since $T$ is nonsingular modulo the good primes, the polynomials $R_i, S_i$ constructed above must agree modulo $p$ with $\overline{R_i}$ and $\overline{S_i}$. Finally we apply a fast interpolation algorithm \cite{Ber-fastmult} to each of the $O((m+n)^2)$ coefficients to reconstruct $\delta$ and all $R_i$, $S_i$ in $(m + n)^2 \beta^{1+\eps}$ bit operations.
\end{proof}

Finally, we mention that we will omit any analysis of the costs of data rearrangement that must be counted in the Turing model; these are all subsumed within the arithmetic cost, along the same lines as the Appendix to \cite{BGS-recurrences}.

\section{An accumulating remainder tree for matrices}

The following is a matrix generalisation of \cite[Theorem~1]{CGH-wilson}.

\begin{prop}
\label{prop:tree}
Let $n \geq 1$, $\lambda \geq 1$ and $B \geq 2$ be integers, and let $\tau \in \RR$, $\tau > 1$. We are given as input a sequence of matrices $M_0, M_1, \ldots, M_{B-1} \in M_n(\ZZ)$, with $\log \norm{M_i} \leq \tau$ for all $i$. Then we may compute
 \[ M_0 M_1 \cdots M_{(p-1)/2} \pmod{p^\lambda} \]
for all primes $3 \leq p < 2B$ simultaneously in
 \[ n^3 (\tau + \lambda) B \log B \log^{1+\eps}(\tau\lambda B) \]
bit operations.
\end{prop}
\begin{proof}
Let $\ell = \lceil\log_2 B\rceil$. We will construct several binary trees of depth $\ell$, whose nodes are indexed by the pairs $(i, j)$ with $0 \leq i \leq \ell$ and $0 \leq j < 2^i$. The root node is $(0, 0)$, the children of $(i, j)$ are $(i + 1, 2j)$ and $(i+1, 2j+1)$, and the leaf nodes are $(\ell, j)$ for $0 \leq j < 2^\ell$.

For each node $(i, j)$ let
 \[ U_{i,j} = \left\{ k \in \ZZ: j \frac{B}{2^i} \leq k < (j+1) \frac{B}{2^i} \right\}. \]
Thus $U_{i,0}, \ldots, U_{i,2^i-1}$ partition the interval $0 \leq k < B$ into $2^i$ sets of roughly equal size. For $0 \leq i < \ell$ we have the disjoint union $U_{i,j} = U_{i+1, 2j} \cup U_{i+1, 2j+1}$. For the leaf nodes, we have $|U_{\ell,j}| = 0$ or $1$ for every $j$, and for every $0 \leq k < B$, there is exactly one $j$ such that $U_{\ell,j} = \{k\}$, namely $j = \lfloor 2^\ell k/B \rfloor$.

Now for each node define
\begin{align*}
 P_{i,j} & = \prod_{\substack{\text{$p$ prime} \\ \frac12(p-1) \in U_{i,j}}} p^\lambda, \\
 A_{i,j} & = \prod_{k \in U_{i,j}} M_{k+1}, \\
 C_{i,j} & = M_0 A_{i,0} A_{i,1} \cdots A_{i,j-1} \pmod{P_{i,j}},
\end{align*}
where for convenience we put $M_B = I$ (the identity matrix). Implicit in the product notation for $A_{i,j}$ is that the $M_k$ are always multiplied in the correct left-to-right order, and that if $U_{i,j} = \emptyset$ then $A_{i,j} = I$.

Note that the desired output may be recovered from the leaf nodes of the $C_{i,j}$ tree. Indeed, suppose that $3 \leq p < 2B$. Let $k = \frac12(p-1)$, and choose $j$ as above so that $U_{\ell,j} = \{k\}$. Then $P_{\ell,j} = p^\lambda$, and $C_{\ell,j} = M_0 M_1 \cdots M_k \pmod{p^\lambda}$.

Now we explain how to compute the values in the trees, beginning with the $P_{i,j}$ tree. After enumerating the primes less than $2B$ in $B \log^{2+\eps} B$ bit operations, we use a standard product tree strategy \cite{Ber-fastmult}, working from the bottom of the tree to the top, using the relation $P_{i,j} = P_{i+1,2j} P_{i+1,2j+1}$. To estimate the complexity, note that $\log P_{i,j} = O(N_{i,j} \lambda \log B)$, where $N_{i,j}$ is the number of primes in $U_{i,j}$, so each product costs $\lambda N_{i,j} \log B \log^{1+\eps}(\lambda N_{i,j} \log B) = \lambda N_{i,j} \log B \log^{1+\eps}(\lambda B)$ bit operations. Since $\sum_j N_{i,j} = \pi(2B) - 1 = O(B / \log B)$, the cost over all intervals at level $i$ is $\lambda B \log^{1+\eps}(\lambda B)$ bit operations. Over all $O(\log B)$ levels of the tree, the cost is $\lambda B \log B \log^{1+\eps}(\lambda B)$ bit operations.

The $A_{i,j}$ tree is computed in a similar manner. We have $\log \norm{A_{i,j}} \leq |U_{i,j}| \tau$ by submultiplicativity. Computing the product $A_{i,j} = A_{i+1,2j} A_{i+1,2j+1}$ requires $O(n^3)$ multiplications of integers with $O(|U_{i,j}| \tau)$ bits, costing $n^3 \tau |U_{i,j}| \log^{1+\eps}(\tau |U_{i,j}|)$ bit operations. The total cost at level $i$ is $n^3 \tau B \log^{1+\eps}(\tau B)$, and the cost over all levels is $n^3 \tau B \log B \log^{1+\eps}(\tau B)$ bit operations.

For the $C_{i,j}$ tree, we work from the top of the tree to the bottom, using the initial condition $C_{0,0} = M_0 \pmod{P_{0,0}}$, and the relations
\begin{align*}
 C_{i+1,2j} & = C_{i,j} \pmod{P_{i+1,2j}}, \\
 C_{i+1,2j+1} & = C_{i,j} A_{i+1,2j} \pmod{P_{i+1,2j+1}}.
\end{align*}
At each node we must perform $n^2$ divisions, and possibly $n^3$ multiplications, of integers with $O(\max(|U_{i,j}|\tau, N_{i,j} \lambda \log B))$ bits. The final cost bound follows by the same argument as the previous paragraphs.
\end{proof}

There are several ways to improve the complexity bound in Proposition \ref{prop:tree}, at the expense of obfuscating the statement of the final result. One could of course substitute a faster matrix multiplication algorithm, such as Strassen's algorithm \cite{Str-gausselim}. This would reduce the exponent of $n$, and hence the exponent of $g$ in Theorem \ref{thm:main}. Another modification, more important in practice, is that one can multiply integer matrices by computing the Fourier transform of the entries, multiplying the matrices of Fourier coefficients, and finally transforming back. The resulting complexity bound depends on what integer multiplication algorithm is being used. For $m$-bit matrix entries, roughly speaking we expect the complexity to drop from $n^3 m \log^{1+\eps} m$ to $n^2 m \log^{1+\eps} m + n^3 m$. For small $n$ and large $m$ the first term dominates. This corresponds to small $g$ and large $N$ in Theorem \ref{thm:main}, and leads to a savings of a factor of $O(g)$ in Theorem \ref{thm:main} as $N \to \infty$.

\section{Reduction towards zero}
\label{sec:reductions}

We now return to cohomology. Define a collection of $\QQ$-subspaces $W_{s,t} \subset \Omega^-$, for $s \geq -1$ and $t \in \ZZ$, as follows. If $s \geq 0$, put
 \[ W_{s,t} = \{F(x) x^s y^{-2t} dx/y : F \in \QQ[x], \deg F \leq 2g\}. \]
For $s = -1$, we use the same definition, but insist that the constant term of $F(x)$ is zero, so that the expression $F(x) x^s y^{-2t} dx/y$ still defines an element of $\Omega^-$.

Our goal in this section is to describe explicit \emph{reduction maps} between the various $W_{s,t}$, that send differentials to cohomologous differentials. The basic building blocks are \emph{horizontal}, \emph{diagonal} and \emph{vertical} reduction maps, that send $W_{s,t}$ to $W_{s-1,t}$, $W_{s-1, t-1}$ and $W_{s,t-1}$ respectively. These maps can be composed to obtain a map from any $W_{s,t}$ to $W_{-1,0}$; by definition this latter map computes the reduction of a differential in $W_{s,t}$, as defined in Section \ref{sec:preliminaries}.

We will represent these maps by $(2g+1) \times (2g+1)$ matrices, acting on coordinate vectors with respect to the natural basis $(x^s y^{-2t} dx/y, \ldots, x^{s+2g} y^{-2t} dx/y)$ for each $W_{s,t}$. In the case $s = -1$, the dimension is only $2g$, but it will be convenient to represent elements of $W_{-1,t}$ as vectors of length $2g + 1$, where it is understood that the first coordinate is always zero. The first row of any matrix mapping into such a space will always be zero.

We will write $\delta \in \ZZ$ for the discriminant of $Q(x)$, or equivalently the resultant of $Q(x)$ and $Q'(x)$. It is nonzero because $Q(x)$ is squarefree. The constant term $c_0$ of $Q(x)$ will also play a special role; some of our results need to be stated slightly differently in the case that $c_0 = 0$.

Our first result is algebraically the same as the `horizontal reduction' discussed in \cite[Prop.~5.4]{Har-kedlaya}. However, we now treat both $s$ and $t$ as variables, and we must analyse coefficient growth, as we are working over $\QQ$ rather than $\QQ_p$.
\begin{lem}[Horizontal reduction]
\label{lem:horizontal}
Let
 \[ D_H(s, t) = (2g+1)(2t-1) - 2s \in \ZZ[s, t]. \]
There exists a matrix $M_H \in M_{2g+1}(\ZZ[s, t])$ with the following properties.
\begin{enumerate}[label=\textup{(\alph*)}]
\item Let $s \geq 0$, $t \in \ZZ$. Then $D_H(s, t) \neq 0$, and the map $D_H(s, t)^{-1} M_H(s, t)$ sends a differential $\omega \in W_{s,t}$ to a cohomologous differential in $W_{s-1, t}$.
\item The entries of $M_H$ have degree at most $1$.
\item $\log \norm{M_H} = O(\log (g\norm{Q}))$.
\item $M_H$ may be computed in $g^{1+\eps} \log^{1+\eps}\norm{Q}$ bit operations.
\end{enumerate}
\end{lem}
\begin{proof}
Using the relations $Q(x) = y^2$ and $Q'(x) dx = 2y\, dy$, we have
\begin{align}
 d(x^s y^{-2t+1}) & = sx^{s-1} y^{-2t+1} dx - (2t - 1) x^s y^{-2t} dy \notag \\
                  & = \left( s Q(x) - \frac12 (2t - 1) x Q'(x) \right) x^{s-1} y^{-2t} dx/y. \label{eq:horizontal}
\end{align}
Let $Q(x) = x^{2g+1} + P(x)$, where $P \in \ZZ[x]$ has degree at most $2g$. After substituting this into the previous equation and rearranging, we obtain
 \[ x^{s+2g} y^{-2t} dx/y \sim \frac{2sP(x) - (2t-1) xP'(x)}{D_H(s, t)} x^{s-1} y^{-2t} dx/y. \]
We may therefore take
\[
M_H = 
\begin{pmatrix}
 0      & 0   & \cdots & 0   & C_0 \\
 D_H    & 0   &        & 0   & C_1 \\
 0      & D_H &        & 0   & C_2 \\
 \vdots & 0   & \ddots &     & \vdots \\
 0      & 0   & \cdots & D_H & C_{2g}
\end{pmatrix},
\]
where $C_i = C_i(s, t)$ is the coefficient of $x^i$ in the polynomial $2sP(x) - (2t-1)x P'(x)$. Note that $D_H(s, t)$ is nonzero for $s, t \in \ZZ$ because it assumes only odd values.

The bound for $\norm{M_H}$ follows from the estimate $\norm{P'} \leq 2g\norm{P}$. The complexity bound covers $O(g)$ multiplications of integers with $O(\log \norm{Q})$ bits by integers with $O(\log g)$ bits.
\end{proof}

Next we give a generalisation of the `vertical reduction' of \cite[Prop.~5.1]{Har-kedlaya}, which was a map $W_{-1,t} \to W_{-1,t-1}$. It turns out that the most natural generalisation yields a map $W_{s,t} \to W_{s-1,t-1}$ rather than $W_{s,t} \to W_{s,t-1}$. (The discrepancy is resolved by reinterpreting the vertical reduction of \cite{Har-kedlaya} as a map from a codimension $1$ subspace of $W_{0,t}$ to $W_{-1,t-1}$.)

\begin{lem}[Diagonal reduction]
\label{lem:diagonal}
Let
 \[ D_D(t) = 2t - 1 \in \ZZ[t]. \]
There exists a matrix $M_D \in M_{2g+1}(\ZZ[s, t])$ with the following properties.
\begin{enumerate}[label=\textup{(\alph*)}]
\item Let $s \geq 0$, $t \in \ZZ$. Then the map $\delta^{-1} D_D(t)^{-1} M_D(s, t)$ sends a differential $\omega \in W_{s,t}$ to a cohomologous differential in $W_{s-1, t-1}$.
\item The entries of $M_D$ have degree at most $1$.
\item $\log|\delta|$ and $\log \norm{M_D}$ are both in $O(g \log(g\norm{Q}))$.
\item $\delta$ and $M_D$ may be computed in $g^{3+\eps} \log^{1+\eps} \norm{Q}$ bit operations.
\end{enumerate}
\end{lem}
\begin{proof}
According to Lemma \ref{lem:bezout}, for each $0 \leq i \leq 2g$, there exist $R_i, S_i \in \ZZ[x]$, with $\deg R_i \leq 2g - 1$ and $\deg S_i \leq 2g$, such that
 \[ \delta x^i = R_i(x) Q(x) + S_i(x) Q'(x). \]
This implies that
\begin{align*}
 \delta x^{s+i} y^{-2t} dx/y & = x^s R_i(x) Q(x) y^{-2t} dx/y + x^s S_i(x) Q'(x) y^{-2t} dx/y \\
 & = x^s R_i(x) y^{-2t+2} dx/y + 2 x^s S_i(x) y^{-2t} dy.
\end{align*}
Since
 \[ d(x^s S_i(x) y^{-2t+1}) = (x^s S_i(x))' y^{-2t+1} dx + (-2t + 1) x^s S_i(x) y^{-2t} dy, \]
after some algebra we obtain the relation in cohomology
\begin{equation}
\label{eq:diagonal}
 x^{s+i} y^{-2t} dx/y \sim \frac{(2t-1) x R_i(x) + 2s S_i(x) + 2x S_i'(x)}{(2t-1)\delta} x^{s-1} y^{-2t+2} dx/y.
\end{equation}
According to this formula, we may take $M_D$ to be the matrix whose $(i+1)$-th column consists of the coefficients of $(2t-1) x R_i(x) + 2s S_i(x) + 2x S_i'(x)$. These coefficients are clearly of degree at most $1$ in $s$ and $t$, and $D_D(t)$ is nonzero for $t \in \ZZ$ because $2t-1$ is odd. This proves (a) and (b), and (c) and (d) follow from Lemma \ref{lem:bezout}.
\end{proof}

We will also need a genuine `vertical reduction' in the generic case $c_0 \neq 0$:
\begin{lem}[Vertical reduction]
\label{lem:vertical}
Assume that $c_0 \neq 0$. Let
 \[ D_V(t) = 2t - 1 \in \ZZ[t]. \]
There exists a matrix $M_V \in M_{2g+1}(\ZZ[s, t])$ with the following properties.
\begin{enumerate}[label=\textup{(\alph*)}]
\item Let $s \geq 0$, $t \in \ZZ$. Then the map $(c_0 \delta)^{-1} D_V(t)^{-1} M_V(s, t)$ sends a differential $\omega \in W_{s,t}$ to a cohomologous differential in $W_{s, t-1}$.
\item The entries of $M_V$ have degree at most $1$.
\item $\log \norm{M_V} = O(g \log(g \norm{Q}))$.
\item $M_V$ may be computed in $g^{3+\eps} \log^{1+\eps}\norm{Q}$ bit operations.
\end{enumerate}
\end{lem}
\begin{proof}
We continue the calculation of Lemma \ref{lem:diagonal}. Write $S_i(x) = h_i + x T_i(x)$, where $h_i \in \ZZ$, $T_i \in \ZZ[x]$, $\deg T_i \leq 2g - 1$. The right hand side of \eqref{eq:diagonal} becomes
 \[ \frac{1}{(2t-1)\delta} \left(2 h_i s x^{s-1} + \big((2t-1) R_i(x) + 2sT_i(x) + 2 S_i'(x)\big) x^s \right) y^{-2t+2} dx/y. \]
Our goal is now to reduce the $x^{s-1} y^{-2t+2} dx/y$ term `to the right'. Write $Q(x) = c_0 + x P(x)$, where $P \in \ZZ[x]$, $\deg P \leq 2g$. Replacing $t$ by $t - 1$ in \eqref{eq:horizontal}, we obtain
 \[ 2 s Q(x) x^{s-1} y^{-2t+2} dx/y \sim (2t - 3) Q'(x) x^s y^{-2t+2} dx/y, \]
so
 \[ 2 s x^{s-1} y^{-2t+2} dx/y \sim \frac{(2t-3) Q'(x) - 2sP(x)}{c_0} x^s y^{-2t+2} dx/y. \]
Combining everything, we finally have
\begin{multline*}
x^{s+i} y^{-2t} dx/y \sim \\
 \frac{(2t - 3) h_i Q' - 2h_i sP + (2t-1) c_0 R_i + 2c_0 s T_i + 2c_0 S_i'}{(2t-1) \delta c_0} x^s y^{-2t+2} dx/y.
\end{multline*}
The columns of $M_V$ are obtained from the numerator of this expression in the same way as in the proof of Lemma \ref{lem:diagonal}.
\end{proof}

The next result has no analogue in \cite{Har-kedlaya}. For each $a$ and $b$, it will allow us to reduce the forms $x^{pa - 1} y^{-pb+1} dx/y \in W_{ap-1, \frac12(bp-1)}$ of Proposition \ref{prop:frob} along the \emph{same} reduction path, for many $p$ simultaneously.

We say that a pair of integers $(a, b)$ is \emph{admissible} if the following conditions hold:
\begin{enumerate}[label={(\roman*)}]
\item $a, b \geq 1$ and $b$ is odd;
\item if $c_0 = 0$, then $b \leq 2a$;
\item $a = O(g^2)$ and $b = O(g)$.
\end{enumerate}
Here the notation $a = O(g^2)$ means that $a \leq Cg^2$ for a suitable absolute constant $C > 0$; an explicit value for $C$ could be extracted from the proof of Theorem \ref{thm:main}. A similar remark applies to $b = O(g)$.
\begin{prop}[Reduction towards zero]
\label{prop:zero}
Let $(a, b)$ be an admissible pair, and let $r \geq 1$. There exists a matrix $M^{a,b}_r \in M_{2g+1}(\ZZ)$ and a nonzero integer $D^{a,b}_r$ with the following properties.
\begin{enumerate}[label=\textup{(\alph*)}]
\item The map $(D^{a,b}_r)^{-1} M^{a,b}_r$ sends a differential $\omega$ in \[ W_{a(2r+1) - 1, \frac12(b(2r+1) - 1)} \] to a cohomologous differential in  \[ W_{a(2r-1) - 1, \frac12(b(2r-1) - 1)}. \]
\item $\log \norm{M^{a,b}_r}$ and $\log \norm{D^{a,b}_r}$ are in $O(g^2 \log(gr \norm{Q}))$.
\item $M^{a,b}_r$ and $D^{a,b}_r$ may be computed in $g^{5+\eps} \log^{1+\eps}(r \norm{Q})$ bit operations.
\end{enumerate}
\end{prop}
\begin{proof}
Our goal is to reduce along the vector $(-2a, -b)$ in the $(s,t)$-plane. We consider two cases.

First suppose that $b \leq 2a$. Then we may construct the required map by performing $b$ diagonal reductions (Lemma \ref{lem:diagonal}) followed by $2a - b$ horizontal reductions (Lemma \ref{lem:horizontal}). More precisely, let
\begin{align*}
 s_0 & = a(2r+1) - 1, \\
 t_0 & = \textstyle \frac12(b(2r+1) - 1), \\
 s_1 = s_0 - b & = a(2r+1) - b - 1, \\
 t_1 = t_0 - b & = \textstyle \frac12 (b(2r-1) - 1), \\
 s_2 = s_1 - (2a - b) & = a(2r-1) - 1, \\
 t_2 = t_1 & = \textstyle \frac12 (b(2r-1) - 1).
\end{align*}
These all have absolute value in $O(ar)$. Let
\begin{align*}
 M' & = M_D(s_0 - b + 1, t_0 - b + 1) \cdots M_D(s_0 - 1, t_0 - 1) M_D(s_0, t_0), \\
 D' & = \delta^b D_D(t_0 - b + 1) \cdots D_D(t_0 - 1) D_D(t_0), \\
 M'' & = M_H(s_1 - 2a + b + 1, t_1) \cdots M_H(s_1 - 1, t_1) M_H(s_1, t_1), \\
 D'' & = D_H(s_1 - 2a + b + 1, t_1) \cdots D_H(s_1 - 1, t_1) D_H(s_1, t_1).
\end{align*}
Then $(D')^{-1} M'$ maps $W_{s_0, t_0}$ to $W_{s_1, t_1}$, and $(D'')^{-1} M''$ maps $W_{s_1, t_1}$ to $W_{s_2, t_2}$. For (a) we should therefore take the composition
 \[ M^{a,b}_r = M'' M', \qquad D^{a,b}_r = D'' D', \]
so that $(D^{a,b}_r)^{-1} M^{a,b}_r$ maps $W_{s_0, t_0}$ to $W_{s_2, t_2}$.

To prove (c), note that for each $0 \leq j < b$, we have $\norm{M_D(s_0 - j, t_0 - j)} = O(ar\norm{M_D})$. Similarly, $\norm{M_H(s_1 - j, t_1)} = O(ar\norm{M_H})$ for $0 \leq j < 2a - b$. Thus
 \[ \log \norm{M^{a,b}_r} = O(b \log(ar \norm{M_D}) + (2a - b) \log(ar \norm{M_H})) = O(g^2 \log(gr \norm{Q})). \]
A similar argument yields $\log \norm{D^{a,b}_r} = O(g^2 \log(gr \norm{Q}))$.

For (d), we may compute $D'$ and $D''$, and hence $D^{a,b}_r$, using a product tree \cite{Ber-fastmult}; the complexity is soft-linear in the number of bits of output, which is $O(g^2 \log(gr\norm{Q}))$. The same result holds for $M^{a,b}_r$, with an additional factor of $O(g^3)$ to account for the matrix multiplications. Therefore we obtain the bit complexity bound $g^{5+\eps} \log^{1+\eps}(r\norm{Q})$. This bound also incorporates the invocations of Lemmas \ref{lem:horizontal} and \ref{lem:diagonal}.

Now consider the case $b > 2a$. By hypothesis we may assume that $c_0 \neq 0$, so that vertical reductions (Lemma \ref{lem:vertical}) are permissible. We proceed by performing $2a$ diagonal reductions followed by $b - 2a$ vertical reductions. In other words, we put
\begin{align*}
 s_0 & = a(2r+1) - 1, \\
 t_0 & = \textstyle \frac12(b(2r+1) - 1), \\
 s_1 = s_0 - 2a & = a(2r-1) - 1, \\
 t_1 = t_0 - 2a & = \textstyle \frac12 (b(2r+1) - 1) - 2a, \\
 s_2 = s_1 & = a(2r-1) - 1, \\
 t_2 = t_1 - (b - 2a) & = \textstyle \frac12 (b(2r-1) - 1),
\end{align*}
and
\begin{align*}
 M' & = M_D(s_0 - 2a + 1, t_0 - 2a + 1) \cdots M_D(s_0 - 1, t_0 - 1) M_D(s_0, t_0), \\
 D' & = \delta^{2a} D_D(t_0 - 2a + 1) \cdots D_D(t_0 - 1) D_D(t_0), \\
 M'' & = M_V(s_1, t_1 - b + 2a + 1) \cdots M_V(s_1, t_1 - 1) M_V(s_1, t_1), \\
 D'' & = (c_0 \delta)^{b - 2a} D_V(t_1 - b + 2a + 1) \cdots D_V(t_1 - 1) D_V(t_1),
\end{align*}
and $M^{a,b}_r = M'' M'$, $D^{a,b}_r = D'' D'$. As before, $(D^{a,b}_r)^{-1} M^{a,b}_r$ maps $W_{s_0, t_0}$ to $W_{s_2, t_2}$, and the required bounds for $\log \norm{M^{a,b}_r}$ and $\log \norm{D^{a,b}_r}$, and the complexity bounds, follow in the same way.
\end{proof}

Iterating the previous result enables us to reduce to $W_{a-1, \frac12(b-1)}$. The next result finishes the job, giving the final reduction to $W_{-1,0}$.
\begin{prop}[Final reduction]
\label{prop:final}
Let $(a, b)$ be an admissible pair. There exists a matrix $M^{a,b}_0 \in M_{2g+1}(\ZZ)$ and a nonzero integer $D^{a,b}_0$ with the following properties.
\begin{enumerate}[label=\textup{(\alph*)}]
\item The map $(D^{a,b}_0)^{-1} M^{a,b}_0$ sends a differential $\omega$ in $W_{a - 1, \frac12(b - 1)}$ to a cohomologous differential in $W_{-1, 0}$.
\item $\log \norm{M^{a,b}_0} = O(g^2 \log(g \norm{Q}))$ and $\log |D^{a,b}_0| = O(g^2 \log(g \norm{Q}))$.
\item $M^{a,b}_0$ and $D^{a,b}_0$ may be computed in $g^{5+\eps} \log^{1+\eps}\norm{Q}$ bit operations.
\end{enumerate}
\end{prop}
\begin{proof}
If $b \leq 2a$, we perform $\frac12(b-1)$ diagonal reductions followed by $a - \frac12(b-1)$ horizontal reductions. If $b > 2a$, we perform $\frac12(b-1) - a$ vertical reductions followed by $a$ diagonal reductions. We omit the details, which are essentially the same as in the proof of Proposition \ref{prop:zero}. 
\end{proof}

\section{The main algorithm}
\label{sec:main}

Recall that $\delta$ denotes the discriminant of $Q(x)$. We say that a pair $(a, b)$ is \emph{$p$-admissible} if it satisfies the following conditions:
\begin{enumerate}[label={(\roman*)}]
\item $a, b \geq 1$ and $b$ is odd;
\item if $p$ divides $c_0$, then $b \leq 2a$;
\item $a = O(g^2)$ and $b = O(g)$;
\item $p$ does not divide $\delta$;
\item $p > (2g+1)b + 2a$.
\end{enumerate}
Note that $p$-admissibility implies admissibility. The following proposition describes how to efficiently compute the forms $U^{a,b}_p$ introduced in Proposition \ref{prop:frob}.
\begin{prop}
\label{prop:main}
Let $(a, b)$ be admissible, and let $N \geq 3$, $\nu \geq 1$, with $\nu = O(g^2)$. Then we may compute $U^{a,b}_p$ modulo $p^\nu$, simultaneously for all those $p < N$ such that $(a, b)$ is $p$-admissible, in
 \[ g^{5+\eps} N \log^2 N \log^{1+\eps}(\norm Q N) \]
bit operations.
\end{prop}
\begin{proof}
We will systematically omit the superscripts $(a,b)$ for clarity. We may assume that $N$ is even, and put $B = N/2$. Let $M_0, \ldots, M_{B-1}$ and $D_0, \ldots, D_{B-1}$ be as in Propositions \ref{prop:zero} and \ref{prop:final}. Then the matrix
 \[ J_p = (D_0 \cdots D_{(p-1)/2})^{-1} (M_0 \cdots M_{(p-1)/2}) \]
maps $W_{ap-1, \frac12(bp-1)}$ cohomologously to $W_{-1,0}$. The form $x^{ap-1} y^{-bp+1} dx/y$ is represented by the vector $(1, 0, \ldots, 0)$ in the source space, so the coordinates of $U_p$ are given by the first column of $J_p$.

To obtain results correct modulo $p^\nu$, we must bound the $p$-adic valuation of $D_0 \cdots D_{(p-1)/2}$. First consider the contributions from the vertical and diagonal reductions. Our hypotheses ensure that the $\delta$ and $c_0$ terms do not contribute. What remains is the factor $2t-1$ for $t = 1, 2, \ldots, \frac12(bp - 1)$. The only such integers divisible by $p$ are $p, 3p, \ldots, (b-2)p$. Since $p > b$, the valuation contributed is exactly $(b-1)/2$.

Now consider the horizontal reductions. If $b > 2a$ then no horizontal reductions are performed, so we may assume that $b \leq 2a$. We must analyse the $p$-adic valuation of $(2g+1)(2t-1) - 2s$ for a certain sequence of pairs $(s, t)$. For all these pairs we have $t \leq \frac12(bp-1)$ and $s \leq ap-1$, so $|(2g+1)(2t-1) - 2s| < p((2g+1)b + 2a) < p^2$. Therefore $(2g+1)(2t-1) - 2s$ cannot be divisible by $p^2$, so it suffices to bound the number of factors $(2g+1)(2t-1) - 2s$ that are divisible by $p$. The pairs coming from the proof of Proposition \ref{prop:zero} are $s = a(2r+1) - b - 1 - j$ and $t = \frac12(b(2r-1)-1)$ for $1 \leq r \leq (p-1)/2$ and $0 \leq j < 2a - b$. For these $s$ and $t$ we have
 \[ (2g+1)(2t - 1) - 2s = 2((2g+1)b - 2a)r - ((2g+1)(b+2) + 2(a - b - 1 - j)). \]
Since $|(2g+1)b - 2a|$ is odd and less than $p$, the coefficient of $r$ is nonzero modulo $p$. Therefore for each $j$, the factor $(2g+1)(2t - 1) - 2s$ is divisible by $p$ for at most one value of $r$. The pairs coming from Proposition \ref{prop:final} are $t = 0$ and $0 \leq s \leq a - 1 -  \frac12(b-1)$. For these pairs we have $|(2g+1)(2t-1) - 2s| \leq 2g+1 + 2a < p$, so they do not contribute any $p$-adic valuation.

We conclude that $v_p(D_0 \cdots D_{(p-1)/2}) \leq \rho$ where $\rho = \frac12(b-1) + \max(0, 2a - b)$. (The `vertical' component of this bound is sharp, but the `horizontal' piece may be too generous by a constant factor. For practical computations it would be important to find the optimal bound, but it does not affect our main asymptotic result.)

We apply Proposition \ref{prop:tree} with $\lambda = \nu + \rho$ to compute the products
 \[ D_0 \cdots D_{(p-1)/2} \pmod{p^\lambda},  \qquad M_0 \cdots M_{(p-1)/2} \pmod{p^\lambda} \]
for all $p < N$. By the above discussion, their ratio yields $J_p$, and hence $U_p$, correctly modulo $p^\nu$, for those $p$ such that $(a, b)$ is $p$-admissible.

Now we analyse the complexity. Each invocation of Proposition \ref{prop:zero} and \ref{prop:final} (i.e.~to compute each $M_r$ and $D_r$) costs $g^{5+\eps} \log^{1+\eps}(N \norm{Q})$ bit operations. There are $O(N)$ such invocations, so the total contribution is $g^{5+\eps} N \log^{1+\eps}(N\norm{Q})$ bit operations. To estimate the contribution from Proposition \ref{prop:tree}, we may take $\tau = \max_r \log\norm{M_r} = O(g^2 \log(gN\norm{Q}))$. Thus the cost of Proposition \ref{prop:tree} is
\[ \begin{split}
 g^3(g^2\log(g N \norm Q) + g^2)N \log N \log^{1+\eps}(g^4 N \log(gN \norm Q)) \\
\begin{aligned}
 & = g^5 N \log N \log(gN \norm Q) \log^{1+\eps}(gN \log(gN \norm Q)) \\
 & = g^5 N \log N \log^{1+\eps}(gN \norm Q) \log^{1+\eps}(gN) \\
 & = g^{5+\eps} N \log^2 N \log^{1+\eps}(N \norm Q). \qedhere
\end{aligned}
\end{split} \]
\end{proof}

Finally we may prove the main theorem.
\begin{proof}[Proof of Theorem \ref{thm:main}]
According to \cite{Ked-hyperelliptic}, the Weil conjectures imply that for each $p$ it suffices to compute the Frobenius matrix modulo $p^{\mu_p}$ where $\mu_p \geq g/2 + (2g+1)\log_p 2$. Therefore the bound $\mu = \lceil g/2 + (2g+1) \log_3 2 \rceil$ works uniformly for all $p$. Note that $\mu = O(g)$.

Consider the terms appearing in the main sum in Proposition \ref{prop:frob}. The corresponding values of $a$ and $b$ satisfy
 \[ 1 \leq a = i + r + 1 \leq (2g - 1) + (2g+1)(\mu - 1) + 1 = (2g+1)\mu - 1 \]
and
 \[ 1 \leq b = 2j + 1 \leq 2\mu - 1. \]
In particular $a = O(g^2)$ and $b = O(g)$.

The definition of $p$-admissiblity requires that $p > (2g+1)b + 2a$, and Proposition \ref{prop:frob} requires that $p > (2g+1)(2\mu - 1)$. Since $(2g+1)b + 2a \leq (2g+1)(4\mu - 1)$, we must first handle separately those $p \leq M$ where $M = (2g+1)(4\mu - 1) = O(g^2)$. This can be done using (for example) Kedlaya's algorithm for each such $p$. The complexity is $p^{1+\eps} g^{4+\eps}$ per prime, and there are $O(g^2)$ such primes, so the total is $g^{8+\eps}$.

Now we use Proposition \ref{prop:main} to compute $U^{a,b}_p \pmod{p^\mu}$, for all pairs $(a, b)$ corresponding to terms appearing in Proposition \ref{prop:frob}. First consider the case $c_0 = 0$. Then we have $C_{j,r} = 0$ for $r < j$, so the relevant pairs are those for which $1 \leq b \leq 2\mu - 1$, $b$ odd, and $\frac12(b+1) \leq a \leq (2g+1)(j+1) - 1$. There are $O(g^3)$ such pairs. All these pairs are admissible, and they are also $p$-admissible for all primes $M < p < N$ of good reduction. The hypotheses of Proposition \ref{prop:main} are satisfied, and we obtain $U^{a,b}_p \pmod{p^\mu}$, for all desired $p$, in $g^{8+\eps} N \log^2 N \log^{1+\eps}(N\norm Q)$ bit operations.

Next consider the case $c_0 \neq 0$. The inequality for $a$ becomes $1 \leq a \leq (2g+1)(j+1) - 1$, and the corresponding pairs are $p$-admissible for all primes $M < p < N$ of good reduction, except those dividing $c_0$. Thus Proposition \ref{prop:main} yields $U^{a,b}_p \pmod{p^\mu}$ for all desired primes except those dividing $c_0$. The number of `missing' primes is $O(\log |c_0|) = O(\log \norm Q)$, and we may handle them separately in $O(g^8 N^{1/2+\eps})$ bit operations each, using the algorithm of \cite{Har-kedlaya}.

At this stage we have computed $U^{a,b}_p \pmod{p^\mu}$, for all relevant pairs $(a, b)$, and for all primes $M < p < N$ of good reduction. The final step is to evaluate the main sum in Proposition \ref{prop:frob}, and compute the characteristic polynomial of the resulting matrix, for each $p$. We will show that this can be achieved in $g^{6+\eps} \log^{1+\eps} p$ bit operations per prime, or $g^{6+\eps} N \log^{1+\eps} N$ bit operations altogether.

We know that $v_p(U^{a,b}_p) \geq -\rho$, where $\rho = O(g^2)$ is defined as in the proof of Proposition \ref{prop:main}, so to evaluate the sum we must work at a $p$-adic precision of $\mu + \rho$ digits. (Numerical evidence suggests that in fact $p U^{a,b}_p$ is always $p$-integral for these primes. A proof can probably be given along the lines of \cite[Lemma 2]{Ked-hyperelliptic}, but we do not need this here.)

We may compute all the $\alpha_j \pmod{p^{\mu+\rho}}$ by a straightforward algorithm, using $O(g^2)$ ring operations (i.e.~operations modulo $p^{\mu + \rho}$), and all the $C_{j,r} \pmod{p^{\mu+\rho}}$ in $g^{3+\eps}$ ring operations. Then for each $0 \leq i < 2g$, we may evaluate the main sum in $O(g^3)$ ring operations, to obtain the reduction $T_i$ of $\sigma_p(x^i dx/y)$ modulo $p^\mu$. Note that the $T_i$ are integral (see for example the proof of \cite[Prop.~4.1]{Har-kedlaya}). The total cost is $O(g^4)$ ring operations, or $g^{6+\eps} \log^{1+\eps} p$ bit operations.

Let $T \in M_{2g}(\ZZ/p^\mu\ZZ)$ be the matrix whose columns are given by the $T_i$; we must compute its characteristic polynomial. We sketch a simple deterministic algorithm for this that avoids divisions by $p$. Compute the powers $T, T^2, \ldots, T^{2g}$. Their traces are the power sums of the eigenvalues of $T$. Newton's identities may be used to deduce the elementary symmetric polynomials in these eigenvalues, and thus the coefficients of the characteristic polynomial. This requires $O(g^4)$ ring operations, including a single division by each of the integers $2, 3, \ldots, 2g$, all of which are less than $p$. The total complexity is $g^{5+\eps} \log^{1+\eps} p$ bit operations.
\end{proof}

\section*{Acknowledgments}

The main ideas for this work arose from a conversation with John Voight. Ian Doust, \'Eric Schost and Andrew Sutherland gave valuable advice on respectively matrix norms, resultant algorithms, and the SEA algorithm. The author thanks the referees and Edgar Costa for suggestions that improved and simplified the presentation. The author was supported by the Australian Research Council, DECRA Grant DE120101293.

\bibliographystyle{amsalpha}
\bibliography{avgpoly}

\providecommand{\bysame}{\leavevmode\hbox to3em{\hrulefill}\thinspace}
\providecommand{\MR}{\relax\ifhmode\unskip\space\fi MR }
\providecommand{\MRhref}[2]{%
  \href{http://www.ams.org/mathscinet-getitem?mr=#1}{#2}
}
\providecommand{\href}[2]{#2}
\begin{thebibliography}{FKRS12}

\bibitem[AH01]{AH-counting}
Leonard~M. Adleman and Ming-Deh Huang, \emph{Counting points on curves and
  abelian varieties over finite fields}, J. Symbolic Comput. \textbf{32}
  (2001), no.~3, 171--189. \MR{1851164 (2002j:14027)}

\bibitem[Ber08]{Ber-fastmult}
Daniel~J. Bernstein, \emph{Fast multiplication and its applications},
  Algorithmic number theory: lattices, number fields, curves and cryptography,
  Math. Sci. Res. Inst. Publ., vol.~44, Cambridge Univ. Press, Cambridge, 2008,
  pp.~325--384. \MR{MR2467550 (2010a:68186)}

\bibitem[BGS07]{BGS-recurrences}
Alin Bostan, Pierrick Gaudry, and {\'E}ric Schost, \emph{Linear recurrences
  with polynomial coefficients and application to integer factorization and
  {C}artier-{M}anin operator}, SIAM J. Comput. \textbf{36} (2007), no.~6,
  1777--1806. \MR{2299425 (2008a:11156)}

\bibitem[BSS00]{BSS-ellcrypt}
I.~F. Blake, G.~Seroussi, and N.~P. Smart, \emph{Elliptic curves in
  cryptography}, London Mathematical Society Lecture Note Series, vol. 265,
  Cambridge University Press, Cambridge, 2000, Reprint of the 1999 original.
  \MR{1771549 (2001i:94048)}

\bibitem[CGH12]{CGH-wilson}
Edgar Costa, Robert Gerbicz, and David Harvey, \emph{A search for {W}ilson
  primes}, to appear in Mathematics of Computation, preprint
  \url{http://arxiv.org/abs/1209.3436}, 2012.

\bibitem[FKRS12]{FKRS-genus2}
Francesc Fit\'e, Kiran~S. Kedlaya, V\'ictor Rotger, and Andrew~V. Sutherland,
  \emph{Sato-{T}ate distributions and {G}alois endomorphism modules in genus
  2}, to appear in Compositio Mathematica, preprint
  \url{http://arxiv.org/abs/1110.6638}, 2012.

\bibitem[GG01]{GG-superelliptic}
Pierrick Gaudry and Nicolas G{\"u}rel, \emph{An extension of {K}edlaya's
  point-counting algorithm to superelliptic curves}, Advances in
  cryptology---{ASIACRYPT} 2001 ({G}old {C}oast), Lecture Notes in Comput.
  Sci., vol. 2248, Springer, Berlin, 2001, pp.~480--494. \MR{1934859
  (2003h:11159)}

\bibitem[GS12]{GS-genus2}
Pierrick Gaudry and {\'E}ric Schost, \emph{Genus 2 point counting over prime
  fields}, J. Symbolic Comput. \textbf{47} (2012), no.~4, 368--400.
  \MR{2890878}

\bibitem[Har07]{Har-kedlaya}
David Harvey, \emph{Kedlaya's algorithm in larger characteristic}, Int. Math.
  Res. Not. IMRN (2007), no.~22, Art. ID rnm095, 29. \MR{MR2376210
  (2009d:11096)}

\bibitem[Har12]{Har-extension}
Michael~C. Harrison, \emph{An extension of {K}edlaya's algorithm for
  hyperelliptic curves}, J. Symbolic Comput. \textbf{47} (2012), no.~1,
  89--101. \MR{2854849}

\bibitem[Ked01]{Ked-hyperelliptic}
Kiran~S. Kedlaya, \emph{Counting points on hyperelliptic curves using
  {M}onsky-{W}ashnitzer cohomology}, J. Ramanujan Math. Soc. \textbf{16}
  (2001), no.~4, 323--338. \MR{MR1877805 (2002m:14019)}

\bibitem[KS08]{KS-L-series}
Kiran~S. Kedlaya and Andrew~V. Sutherland, \emph{Computing {$L$}-series of
  hyperelliptic curves}, Algorithmic number theory, Lecture Notes in Comput.
  Sci., vol. 5011, Springer, Berlin, 2008, pp.~312--326. \MR{2467855
  (2010d:11070)}

\bibitem[KS09]{KS-hyperelliptic}
\bysame, \emph{Hyperelliptic curves, {$L$}-polynomials, and random matrices},
  Arithmetic, geometry, cryptography and coding theory, Contemp. Math., vol.
  487, Amer. Math. Soc., Providence, RI, 2009, pp.~119--162. \MR{2555991
  (2011d:11154)}

\bibitem[Min10]{Min-superelliptic}
Moritz Minzlaff, \emph{Computing zeta functions of superelliptic curves in
  larger characteristic}, Math. Comput. Sci. \textbf{3} (2010), no.~2,
  209--224. \MR{2608297}

\bibitem[MW68]{MW-formal-I}
P.~Monsky and G.~Washnitzer, \emph{Formal cohomology. {I}}, Ann. of Math. (2)
  \textbf{88} (1968), 181--217. \MR{MR0248141 (40 \#1395)}

\bibitem[Pap94]{Pap-complexity}
Christos~H. Papadimitriou, \emph{Computational complexity}, Addison-Wesley
  Publishing Company, Reading, MA, 1994. \MR{1251285 (95f:68082)}

\bibitem[Pil90]{Pil-abelian}
J.~Pila, \emph{Frobenius maps of abelian varieties and finding roots of unity
  in finite fields}, Math. Comp. \textbf{55} (1990), no.~192, 745--763.
  \MR{1035941 (91a:11071)}

\bibitem[Sch85]{Sch-elliptic}
Ren{\'e} Schoof, \emph{Elliptic curves over finite fields and the computation
  of square roots mod {$p$}}, Math. Comp. \textbf{44} (1985), no.~170,
  483--494. \MR{777280 (86e:11122)}

\bibitem[Str69]{Str-gausselim}
Volker Strassen, \emph{Gaussian elimination is not optimal}, Numer. Math.
  \textbf{13} (1969), 354--356. \MR{MR0248973 (40 \#2223)}

\bibitem[Sut12]{Sut-modular}
Andrew~V. Sutherland, \emph{On the evaluation of modular polynomials}, preprint
  \url{http://arxiv.org/abs/1202.3985}, 2012.

\bibitem[vzGG03]{vzGG-compalg}
Joachim von~zur Gathen and J{\"u}rgen Gerhard, \emph{Modern computer algebra},
  second ed., Cambridge University Press, Cambridge, 2003. \MR{2001757
  (2004g:68202)}

\bibitem[Wan08]{Wan-zeta}
Daqing Wan, \emph{Algorithmic theory of zeta functions over finite fields},
  Algorithmic number theory: lattices, number fields, curves and cryptography,
  Math. Sci. Res. Inst. Publ., vol.~44, Cambridge Univ. Press, Cambridge, 2008,
  pp.~551--578. \MR{2467557 (2010c:11157)}

\end{thebibliography}

\end{document}